\documentclass[12pt]{amsart}
\usepackage{amsmath,latexsym,amscd,amsbsy,amssymb,amsfonts,amsthm,fleqn,leqno,
euscript, graphicx, texdraw, pb-diagram}
\usepackage[all]{xy}
\usepackage{color}
\usepackage{tikz}
\numberwithin{equation}{section}
\newtheorem{thm}{Theorem}[section]
\newtheorem{pro}[thm]{Proposition}
\newtheorem{lem}[thm]{Lemma}
\newtheorem{cor}[thm]{Corollary}

\newtheorem{qu}[thm]{Question}

\theoremstyle{definition}

\theoremstyle{remark}
\newtheorem{claim}[thm]{Claim}

\begin{document}
%%%%%%%%%%% Begin Topmatter %%%%%%%%%%%%%%%%%

\title[Extending homeomorphisms on Cantor cubes ]
{Extending homeomorphisms on Cantor cubes }

\author{E. Shchepin}
\address{Steklov Mathematical Institute of Russian Academy of Sciences,
8 Gubkina St. Moscow, 119991, Russia}
\email{scepin@yandex.ru}

\author{V. Valov}
\address{Department of Computer Science and Mathematics,
Nipissing University, 100 College Drive, P.O. Box 5002, North Bay,
ON, P1B 8L7, Canada} \email{veskov@nipissingu.ca}

\thanks{The second author was partially supported by NSERC
Grant 261914-19.}



%%%%%%%%%% End topmatter %%%%%%%%%%%%%%%%%%%%%
%%%%%%%%%% End topmatter %%%%%%%%%%%%%%%%%%%%%

\begin{abstract}
We discuss the question of extending homeomorphism between closed subsets of the Cantor discontinuum $D^\tau$. For every set $P\subset D^\tau$ let $\mathfrak{L}_P$ be the set of cardinality $\lambda$ such that the $\lambda$-interior of $P$ is not empty. It is established that any homeomorphism $f$ between two proper closed subsets $P$ and $K$ of $D^\tau$ can be extended to an autohomeomorphism of $D^\tau$ provided the sets 
$\mathfrak{L}_P$ and $\mathfrak{L}_K$ do not have so many points of discontinuity and 
$f$ preserves the $\lambda$-interiors of $P$ and $K$.
\end{abstract}

\maketitle
\markboth{}{Extending homeomorphisms}

%%%%%%%%%%%%%%%%%%%%%%%%%%%%%%%%%%%%%%%%%%%%%%%%%%%%%%%%%%%%%%%%
%%%%%%%%%%%%%%%%%%%%%%%%%%%%%%%%%%%%%%%%%%%%%%%%%%%%%%%%%%%%%%%%

%%%%%%%%%%%%%%%%%% TABLE OF CONTENT %%%%%%%%%%%%%%%%%%%%%%%%%%%%%%%%%%

%\tableofcontents
%\section{Introduction}

\section{Introduction}

According to \cite{kr}, in 1951 Ryll-Nardzewski presented his solution of Knaster's problem on extension of homeomorphisms on closed subsets of the Cantor set $D^{\aleph_0}$. Ryll-Nardzewski's proof was based on Boolean algebras and it was not published.
%A topological proof of the Ryll-Nardzewski theorem was given by Knaster-Reichbach \cite{kr}.

Here is Ryll-Nardzewski's theorem, see \cite{kr}: \textit{Let $P, K$ be a proper closed subsets of the Cantor set $D^{\aleph_0}$ and $f$ be a homeomorphism between $P$ and $K$ such that $f(\rm{int}~P)=\rm{int}~K$. Then there exists an autohomeomorphism of $D^{\aleph_0}$ extending $f$.}
A topological proof of Ryll-Nardzewski's theorem was established by Knaster-Reichbach \cite{kr}. More precisely, Knaster-Reichbach reduced their proof to the case when $\rm{int}~P=\rm{int}~K=\varnothing$, see Theorem 2.6 below. The non-metrizable analogue of Knaster-Reichbach's theorem was proved in our paper
\cite[Theorem 1.2]{sv}.

The present article is devoted to the proof of a non-metrizable analogue of the Ryll-Nardzewski theorem. In contrast to the metrizable case, the non-metrizable version of Ryll-Nardzewski's theorem cannot be directly obtained from \cite[Theorem 1.2]{sv}. The key role in our proof plays the concept of $\lambda$-interior.

For a space $X$, a subset $P\subset X$ and an infinite cardinal $\lambda$ we denote by $P^{(\lambda)}$ the {\em $\lambda$-interior} of $P$ in $X$, i.e. the set all $x\in P$ such that there exists a $G_\lambda$-subset $K$ of $X$ with $x\in K\subset P$. If $\lambda$ is finite, then $P^{(\lambda)}$ is the ordinary interior of $P$ and it is denoted by $P^{(0)}$. If there exists $\tau\geq\aleph_0$ such that
$P^{(\lambda)}$ is empty for all $\lambda<\tau$, we say that $P$ is {\em $\tau$-negligible in $X$}. The sets of all cardinals $\lambda$ with $P^{(\lambda)}\neq\varnothing$ is denoted by $\mathfrak{L}_P$.

Let $\lambda\in\mathfrak{L}_P$ be a limit cardinal. We say that $\mathfrak{L}_P$ is {\em discontinuous at $\lambda$} if 
$\overline{\bigcup_{\gamma<\lambda}P^{(\gamma)}}\subsetneqq P^{(\lambda)}$, otherwise $\mathfrak{L}_P$ is said to be {\em continuous at $\lambda$}. Denote by $\mathfrak{L}_P^d$ the cardinals of discontinuity  for $\mathfrak{L}_P$. The set $\mathfrak{L}_P^d$ is said to be discrete in $\mathfrak{L}_P$ if for every $\lambda\in \mathfrak{L}_P^d$ there is $\gamma<\lambda$ such that the interval $(\gamma,\lambda)$ is disjoint from $\mathfrak{L}_P^d$. Obviously, $\mathfrak{L}_P^d$ is discrete in $\mathfrak{L}_P$ if $\mathfrak{L}_P^d$ is finite, in particular that is true
provided $\mathfrak{L}_P$ is finite.
 
Now, we can formulate a non-metrizable analogue of Ryll-Nardzewski's theorem.
\begin{thm}
Let $f$ be a homeomorphism between two proper closed subsets $P$ and $K$  of $D^\tau$ such that $\mathfrak{L}_P^d$ is discrete in $\mathfrak{L}_P$ and $f(P^{(\lambda)})=K^{(\lambda)}$ for  every $\lambda\in\mathfrak{L}_P$. Then $f$ can be extended to a homeomorphism of $D^\tau$.
\end{thm}
Observe that the condition $f(P^{(\lambda)})=K^{(\lambda)}$ for any cardinal number $\lambda\geq 0$ is necessary for the
existence of a homeomorphic extension of $f$. Let's also note that the case $P^{(\lambda)}=\varnothing$ for all $\lambda<\tau$ was settled in \cite{sv}.
\begin{qu}
Let $f$ be a homeomorphism between two proper closed subsets $P$ and $K$  of $D^\tau$. Is it true that $f$ can be extended to a homeomorphism of $D^\tau$ provided $f(P^{(\lambda)})=K^{(\lambda)}$ for  every $\lambda\in\mathfrak{L}_P$? 
\end{qu}

\section{Some preliminary results}
Everywhere below, if not stated otherwise,
$X=\prod_{\alpha\in A}X_\alpha$ is an uncountable product of metric compacta, $P\subset X$ is a closed proper subset and $f:P\to P$ is a homeomorphism such that $f(P^{(\mu)})=P^{(\mu)}$ for  every cardinal $\mu$. If $B\subset A$ we denote by $\pi_{B}:X\to\prod_{\alpha\in B}X_\alpha$ the projection. In case $B=\{\alpha\}$ we write $\pi_\alpha$ instead of $\pi_{\{\alpha\}}$. We also denote $\pi_B(P)$ by $P_B$ and $\pi_B(P^{(\lambda)})$ by $P_B^{(\lambda)}$. We have to distinguish the sets $P_B^{(\lambda)}$ and $(P_B)^{(\lambda)}$.
Since  $P^{(\mu)}$ is a union of $G_\mu$-subsets of $X$, according to \cite{ef} the closure of $P^{(\mu)}$ is also a $G_\mu$-set. Hence,
$P^{(\mu)}$ is closed in $X$ for every cardinal $\mu$ and there is a set $B\subset A$ of cardinality $\mu$ with $\pi_B^{-1}(P_B^{(\mu)})=P^{(\mu)}$.
A set $B\subset A$ is called {\em $f$-admissible} if there exists a homeomorphism $f_B:P_B\to P_B$ with
$\displaystyle\pi_{B}\circ f=f_B\circ\pi_{B}$. It is easily seen that an arbitrary union of $f$-admissible sets is also $f$-admissible.
According to \cite[Proposition 2.1]{sv} there is a family $\{B(\alpha):\alpha\in A\}$ of countable $f$-admissible sets such that $\alpha\in B(\alpha)$ for every $\alpha$. We say that a set $C\subset A$ is {\em saturated} if $C=\bigcup_{\alpha\in C}B(\alpha)$. Every saturated set is
$f$-admissible.

\begin{pro}
Every set $C\subset A$ is contained in a saturated set $B\subset A$ of the same cardinality as $C$.
\end{pro}

\begin{proof}
Since the family of $f$-admissible sets is closed under arbitrary unions, the set
$B=\bigcup_{\alpha\in C}B(\alpha)$ is saturated and has the same cardinality as $C$.
\end{proof}

\begin{lem}\cite[Lemma 2.2]{sv}
Let $K\subset X$ be a closed set.
Suppose $\tau >\aleph_0$ and $C\subset A$ is a set of cardinality $<\tau$ such that $(\{z\}\times X_{A\backslash C})\cap K$ is $\tau$-negligible in $\{z\}\times X_{A\backslash C}$ for every $z\in K_C$. Then $K_{A\backslash C}\neq X_{A\backslash C}$.
\end{lem}

If $\lambda$ is any cardinal number, then $\lambda^+$ denotes the successor of $\lambda$.

\begin{pro}\label{pro}
Let $C\subset A$ is a saturated set of cardinality $\lambda\geq\aleph_0$ such that $\pi_C^{-1}(P_C^{(\lambda)})=P^{(\lambda)}$.
If $P\neq P^{(\lambda)}$
then there is a saturated set $B\subset A$ of cardinality $\lambda^+$  such that
\begin{itemize}
\item[(i)] $C\subset B$ and $\pi_B^{-1}(P_B^{(\lambda^+)})=P^{(\lambda^+)}$;
\item[(ii)] Every compact set $F\subset P_B\backslash P_B^{(\lambda)}$ is  $\lambda^+$-negligible in $X_B$.
\end{itemize}
\end{pro}
\begin{proof}
Since $P_C^{(\lambda)}$ is a closed subset of $P_C$ and the weight of $X_C$ is $\lambda$, $P_C\backslash P_C^{(\lambda)}$ is the union of
$\lambda$ many closed sets $F_\xi$, $\xi<\omega(\lambda)$ (here $\omega(\lambda)$ is the first ordinal of cardinality $\lambda$). Let $P_\xi=P\cap\pi_C^{-1}(F_\xi)$. Because $P_\xi$ are disjoint from $P^{(\lambda)}$ and $P\neq P^{(\lambda)}$,
%$P^{(\beta)}= P^{(\lambda)}$ for every $\beta$ with $\lambda\leq\beta<\mu$,
each $P_\xi$ is
a $\lambda^+$-negligible set in $X$. So are the sets $P_\xi(y)=P_\xi\cap(\{y\}\times X_{A\backslash C})$, $y\in F_\xi$. This observation implies that
%$\pi_B(P_\xi(y))$ is $\mu$-negligible in $X_B$.
each $\pi_{A\backslash C}(P_\xi(y))$ is $\lambda^+$-negligible in $X_{A\backslash C}$. Indeed, otherwise $\pi_{A\backslash C}(P_\xi(y))$ would contain a closed $G_\nu$-subset of $X_{A\backslash C}$ with $\nu\leq\lambda$, and using that $\{y\}\times X_{A\backslash C}$ is a $G_\lambda$-set in $X$, one can show that $P_\xi(y)$ also contains a closed $G_\lambda$-subset of $X$. This is a contradiction because
$P_\xi(y)\subset P\backslash P^{(\lambda)}$.

Therefore, according to Lemma 2.2,
$\pi_{A\backslash C}(P_\xi)\neq X_{A\backslash C}$ for every $\xi$. Hence, we can find a countable set $C_1(\xi)\subset A\backslash C$
such that $\pi_{C_1(\xi)}(P_\xi)\neq X_{C_1(\xi)}$. By Proposition 2.1, we can assume that each $C_1(\xi)$ is saturated.
 Then $C_1=\bigcup_\xi C_1(\xi)$ is a saturated set of cardinality $\lambda$ disjoint from $C$ with
$\pi_{C_1}(P_\xi)\neq X_{C_1}$ for all $\xi$. Moreover, $B_1=C\cup C_1$ is also a saturated set of cardinality $\lambda$, and let $F_\xi^1=(\pi^{B_1}_C)^{-1}(F_\xi)$. Obviously,
$P_\xi=P\cap\pi_{B_1}^{-1}(F_\xi^1)$,
$\pi_{B_1}^{-1}(P_{B_1}^{(\lambda)})=P^{(\lambda)}$ and $P_{B_1}\backslash P_{B_1}^{(\lambda)})=\bigcup_\xi F_\xi^1$. Since all sets $P_\xi(z)=P_\xi\cap(\{z\}\times X_{A\backslash B_1})$, $z\in F_\xi^1$, are $\lambda^+$-negligible in $X$, the arguments from the previous paragraph show that the sets
$\pi_{A\backslash B_1}(P_\xi(z))$ are $\lambda^+$-negligible in $X_{A\backslash B_1}$.
So, by Lemma 2.2  $\pi_{A\backslash B_1}(P_\xi)\neq X_{A\backslash B_1}$ for every $\xi$. As above, we can find a family  of countable saturated sets  $C_2(\xi)\subset A\backslash B_1$ with $\pi_{C_2(\xi)}(P_\xi)\neq X_{C_2(\xi)}$.
Let $C_2=\bigcup_\xi C_2(\xi)$. In this way, by transfinite induction we construct for every $\gamma<\omega(\lambda^+)$ a family
$\phi_\gamma=\{C_\gamma(\xi):\xi<\omega(\lambda)\}$ of countable saturated sets such that the sets
$C_\gamma=\bigcup_{\xi}C_\gamma(\xi)$ satisfy the following conditions :
\begin{itemize}
\item[(1)] $\{C_\gamma:\gamma<\omega(\lambda^+)\}$ is a disjoint family;
\item[(2)] $C_\gamma\subset A\backslash\bigcup_{\beta<\gamma}(C\cup C_\beta)$ if $\gamma$ is a limit ordinal;
\item[(3)] $C_{\gamma+1}\subset A\backslash\bigcup_{\beta\leq\gamma}(C\cup C_\beta)$;
\item[(4)] $\pi_{C_\gamma}(P_\xi)\neq X_{C_\gamma}$ for all $\xi$.
\end{itemize}
Let $B'=\bigcup_\gamma B_\gamma$, where $B_\gamma=\bigcup_{\beta\leq\gamma}C\cup C_\beta$. Clearly, $B'$ is a saturated set of cardinality  $\lambda^+$ and $C\subset B'$.
Since $P^{(\lambda^+)}$ is a closed $G_{\lambda^+}$-subset of $X$, there is a saturated set $B\subset A$ of cardinality $\lambda^+$ containing $B'$ such that
$\pi_B^{-1}(P_B^{(\lambda^+)})=P^{(\lambda^+)}$.

%Suppose that $D\subset A$ is a set containing $B$.
For every $\xi<\omega(\lambda)$ let $H_\xi=(\pi_C^B)^{-1}(F_\xi)\cap P_B$. We have $\pi_B^{-1}(P_B^{(\lambda)})=P^{(\lambda)}$,
$P_B\backslash P_B^{(\lambda)}=\bigcup_\xi H_\xi$ and $\pi^B_{C_\gamma}(H_\xi)=\pi_{C_\gamma}(P_\xi)$ for all $\gamma$.
\begin{claim}
Every compact set $F\subset P_B\backslash P_B^{(\lambda)}$ is $\lambda^+$-negligible in $X_B$.
\end{claim}
%Since $P^{(\beta)}\backslash P^{(\lambda)}=\varnothing$ for every $\beta$ with $\lambda\leq\beta<\mu$, it suffices to show that every
%$P_\xi$ is a $\lambda^+$-negligible set in $X$.
Fix a point $y_C^*\in X_C\backslash P_C^{(\lambda)}$ and for every $\gamma<\omega(\lambda^+)$ choose a point $y_\gamma^*\in X_{C_\gamma}\backslash\pi_{C_\gamma}(P_\xi)$.
Let $y^*\in X_B$ be a point with $\pi_C^B(y^*)=y_C^*$ and $\pi_{C_\gamma}^B(y^*)=y_{\gamma}^*$ for all $\gamma<\omega(\lambda^+)$. Consider the set $\sum(y^*)$ of all
$y\in X_B$ such that the cardinality of the set $\{\gamma<\omega(\lambda^+):\pi^B_{C_\gamma}(y)\neq y^*_\gamma)\}$ is $<\lambda^+$.
Item $(4)$ implies $H_\xi\subset X_B\backslash\sum(y^*)$ for all $\xi$, so $\bigcup_{\xi}H_\xi\subset X_B\backslash\sum(y^*)$.
Since $P_B\backslash P_B^{(\lambda)}=\bigcup_\xi H_\xi$, $F\subset X_B\backslash\sum(y^*)$.
Suppose $F$ contains a closed $G_\beta$-subset of $X_B$ for some $\beta<\lambda^+$. Then there exist a set $\Gamma\subset B$ of cardinality $\beta$ and a point $z\in X_\Gamma$  with
$(\pi^B_\Gamma)^{-1}(z)=\{z\}\times X_{B\backslash\Gamma}\subset F$. Since $\Gamma$ can contains at most $\beta$ many sets $C_\gamma$,
%Then $A\backslash\Gamma$ contains $\lambda^+$-many sets $C_\gamma$. So by $(1)$,
$\{z\}\times X_{B\backslash\Gamma}$ contains points from
$\sum(y^*)$, which contradicts $F\subset X_B\backslash\sum(y^*)$.  Hence, $F$ is $\lambda^+$-negligible in $X_B$.
\end{proof}

\begin{lem}
Let $C\subset B$ be saturated sets satisfying the hypotheses of Proposition $2.3$.
Then for any $\sigma$-compact set
$F\subset P_C\backslash P_C^{(\lambda)}$ there exists a saturated set $\Lambda\subset B$ containing $C$ such that $\Lambda\backslash C$ is countable and for all $x\in F$ the sets
$\pi^\Lambda_{\Lambda\backslash C}((\pi_C^\Lambda)^{-1}(x)\cap P_\Lambda)$, $\pi^\Lambda_{\Lambda\backslash C}((\pi_C^\Lambda)^{-1}(f_C(x))\cap P_\Lambda)$
are nowhere dense in $X_{\Lambda\backslash C}$. Moreover, we can suppose that $\Lambda$ contains $\Gamma$ for a given saturated set $\Gamma\subset B$ containing $C$ with $\Gamma\backslash C$ countable.
\end{lem}

%\textcolor{red}
\begin{proof}
Let $F'=f_C(F)$. Since $C$ is $f$-admissible, $f_C(P_C^{(\lambda)})=P_C^{(\lambda)}$. Hence, $F\cup F'$ is a $\sigma$-compact subset of $P_C\backslash P_C^{(\lambda)}$. Represent $F\cup F'$ as a countable union of compact sets $F_n\subset P_C\backslash P_C^{(\lambda)}$.
Suppose that $\Gamma\subset B$ is a saturated set containing $C$ with $\Gamma\backslash C$ countable. Then each
$K_n=(\pi_C^\Gamma)^{-1}(F_n)\cap P_\Gamma$ is a compact subset of $P_\Gamma\backslash P_\Gamma^{(\lambda)}$.
Consequently, all $L_n=(\pi_\Gamma^B)^{-1}(K_n)\cap P_B$ are contained in $P_B\backslash P_B^{(\lambda)}$. We claim that for every $n$ and $z\in K_n$ the sets $\pi^B_{B\backslash\Gamma}(L_n(z))$ are $\lambda^+$-negligable subsets of $X_{B\backslash\Gamma}$, where $L_n(z)=(\{z\}\times X_{B\backslash\Gamma})\cap L_n$.
Indeed, otherwise some $\pi^B_{B\backslash\Gamma}(L_n(z))$ would contain a closed $G_\nu$-subset of $X_{B\backslash\Gamma}$ with $\nu<\lambda^+$, and using that $\{z\}\times X_{B\backslash\Gamma}$ is a $G_\lambda$-set in $X_B$ (see the proof of Proposition 2.3), one can show that $L_n(z)$ contains a closed $G_\lambda$-subset of $X_B$. Hence, $L_n(z)\subset P_B^{(\lambda)}$, which contradicts the inclusion  $L_n(z)\subset L_n\subset P_B\backslash P_B^{(\lambda)}$.

Therefore, according to Lemma 2.2,
$\pi^B_{B\backslash\Gamma}(L_n)\neq X_{B\backslash\Gamma}$ for every $n$, and there exist countable sets $B_1(n)\subset B\backslash\Gamma$
such that $\pi^B_{B_1(n)}(L_n)\neq X_{B_1(n)}$. By Proposition 2.1, we can assume that each $B_1(n)$ is saturated.
 Then $B_1=\bigcup_n B_1(n)$ is a countable saturated subset of $B\backslash\Gamma$ with
$\pi^B_{B_1}(L_n)\neq X_{B_1}$ for all $n$. Moreover, $\Lambda(1)=\Gamma\cup B_1$ is also a countable saturated set.

Now, let $K_n^1=(\pi^{\Lambda(1)}_\Gamma)^{-1}(K_n)\cap P_{\Lambda(1)}$ and $L_n^1(z)=(\{z\}\times X_{B\backslash\Lambda(1)})\cap L_n$ for every
$z\in K_n^1$. As above, we can show that $\pi^B_{B\backslash\Lambda(1)}(L_n^1(z))$ are $\lambda^+$-negligible subsets of $X_{B\backslash\Lambda(1)}$ for all $n$ and $z\in K_n^1$. Since $\pi^B_{\Lambda(1)}(L_n)=K_n^1$, we can apply Lemma 2.2 to conclude that
$\pi^B_{B\backslash\Lambda(1)}(L_n)\neq X_{B\backslash\Lambda(1)}$ for every $n$. Hence, there exist a sequence $\{B_2(n)\}$ of countable saturated subsets of $B\backslash\Lambda(1)$ with $\pi^B_{B_2(n)}(L_n)\neq X_{B_2(n)}$. This implies $\pi^B_{B_2}(L_n)\neq X_{B_2}$ for all $n$, where $B_2=\bigcup_n B_n(n)$. Then $\Lambda(2)=\Gamma\cup B_1\cup B_2$ is a saturated subset of $B$ such that $B_2$ and $B_1$ are disjoint subsets of $B\backslash\Gamma$. In this way we construct a sequence $\{B_k\}$ of disjoint countable saturated subsets of $B\backslash\Gamma$ satisfying the following conditions:
\begin{itemize}
\item[(a)] $B_{k+1}\subset B\backslash\bigcup_{i\leq k}(\Gamma\cup B_i)$;
\item[(b)] $\pi^B_{B_k}(L_n)\neq X_{B_k}$ for all $n$ and $k$.
\end{itemize}
Let $\Lambda$ be the union of $\Gamma$ and all $B_k$. For every $n$ we have $L_n=(\pi^B_C)^{-1}(F_n)\cap X_B$. So,
$\pi^B_\Lambda(L_n)=(\pi^\Lambda_C)^{-1}(F_n)\cap P_\Lambda$ and
$\pi^\Lambda_{\Lambda\backslash C}((\pi_C^\Lambda)^{-1}(F_n)\cap P_\Lambda)=\pi^B_{\Lambda\backslash C}(L_n)$. Suppose  $\pi^B_{\Lambda\backslash C}(L_n)$ contains an open subset of $X_{\Lambda\backslash C}$ for some $n$.
Since $B_k$ are disjoint subsets of $\Lambda\backslash C$, there exists $k$ with
$\pi^{\Lambda\backslash C}_{B_k}(\pi^B_{\Lambda\backslash C}(L_n))=X_{B(k)}$. On the other hand, $\pi^{\Lambda\backslash C}_{B_k}(\pi^B_{\Lambda\backslash C}(L_n))=\pi^B_{B_k}(L_n)$, which contradicts condition $(b)$. Therefore, $\pi^\Lambda_{\Lambda\backslash C}((\pi_C^\Lambda)^{-1}(F_n)\cap P_\Lambda)$ are nowhere dense subsets of $X_{\Lambda\backslash C}$.
This implies that the sets $\pi^\Lambda_{\Lambda\backslash C}((\pi_C^\Lambda)^{-1}(x)\cap P_\Lambda)$, $x\in F\cup F'$, are also nowhere dense in $X_{\Lambda\backslash C}$.
\end{proof}

\begin{thm}\cite{kr}
Let $X$ and $Y$ be compact, perfect zero-dimensional metric spaces, and let $P$ and $K$ be closed nowhere dense subsets of $X$ and $Y$, respectively. If $f$ is a homeomorphism between $P$ and $K$, then there exists a homeomorphism between $X$ and $Y$ extending $f$.
\end{thm}

\section{Homeomorphisms on product spaces}

For any space $X$ let $\mathcal H(X)$ denote the set of all homeomorphisms of $X$ equipped with the compact-open topology. In this section we prove that $\mathcal H(X)$, where $X$ is a product of compact metric spaces, is an absolute extensor for compact 0-dimensional spaces. This is easily seen (see the proof of Theorem 3.4 below) when $X$ is a countable product. So, the interesting case is when $X$ is an uncountable product of metric compacta.
We use the technique developed in \cite{m}.

\begin{pro}\label{hom}
Let $K\subset\mathcal H(X)$ be a Lindel$\ddot{o}$f subset, where $X=\prod_{\alpha\in A}X_\alpha$ is a product of compact metric spaces with $|A|=\tau>\aleph_0$. Then
$A$ can be covered by a family of sets $\{A(\alpha):\alpha\in\omega(\tau)\}$ such that for every $\alpha$ we have:
\begin{itemize}
\item $A(\alpha)=\bigcup_{\gamma<\alpha}A(\gamma)$ if $\alpha$ is a limit ordinal;
\item $A(\alpha)\subset A(\alpha+1)$ and $A(\alpha+1)\backslash A(\alpha)$ is countable for all $\alpha$;
\item For every $f\in K$ and $\alpha\in A$ there is $f_\alpha\in\mathcal H(X_{A(\alpha)})$ with $\pi_{A(\alpha)}\circ f=f_\alpha\circ\pi_{A(\alpha)}$.
\end{itemize}
\end{pro}
\begin{proof}
Let $B\subset A$ be a countable set. Take a sequence of open covers $\{\mathcal U_n\}_{n\geq 1}$ of $X_B$ such that $\rm{diam}(\mathcal U_n)<1/n$ for all $n$. Since $X_B$ is metrizable and $X$ is compact, the compact-open topology on the function space $C(X,X_B)$ coincides with the limitation topology \cite{to}. Recall that
$U\subset C(X,X_B)$ is open with respect to the limitation topology if for every $f\in U$ there is $\mathcal V\in\rm{cov}(X_B)$  such that $U$ contains the set
$B(f,\mathcal V)=\{g\in C(X,X_B):g{~}\mbox{is}{~}\mathcal V-\mbox{close to}{~}f\}.$ Here, $\rm{cov}(X_B)$ is the family of all open covers of $X_B$ and $g$ is $\mathcal V$-close $f$ provided for any $x\in X$ there is $V\in\mathcal V$ containing both points $f(x)$ and $g(x)$. In particular, every $B(f,\mathcal V)$ contains a neighborhood $B_*(f,\mathcal V)$ of $f$, see \cite{bo1}. There exists a natural map $p_B:\mathcal H(X)\to C(X,X_B)$,
$p_B(h)=\pi_B\circ h$, which is continuous when both $\mathcal H(X)$ and $C(X,X_B)$ carry the compact-open topology.

\begin{claim}\label{1}
There is a countable set $\Gamma(B)\subset A$ containing $B$ such that for every $f\in K$ there exist
$f_{\Gamma(B)},g_{\Gamma(B)}\in C(X_{\Gamma(B)},X_B)$ with $\pi_{B}\circ f=f_{\Gamma(B)}\circ\pi_{\Gamma(B)}$ and
$\pi_{B}\circ f^{-1}=g_{\Gamma(B)}\circ\pi_{\Gamma(B)}$.
\end{claim}
Since for each $n$ the family $\{B_*(\pi_{B}\circ f,\mathcal U_n):f\in K\}$ is an open cover of $p_B(K)$, there is a sequence $\{f_{ni}\}_{i\geq 1}\subset\mathcal H(X)$ such that
$\{B_*(\pi_{B}\circ f_{ni},\mathcal U_n):i\geq 1\}$ covers $p_B(K)$. Because $\mathcal H(X)$ is a topological group, the set $K^{-1}=\{f^{-1}:f\in K\}$ is also Lindel$\ddot{o}$f. Hence,
there exists a sequence $\{g_{ni}\}_{i\geq 1}\subset\mathcal H(X)$ for each $n\geq 1$ such that
$\{B_*(\pi_{B}\circ g_{ni},\mathcal U_n):i\geq 1\}$ covers $p_B(K^{-1})$.
Because every continuous function on $X$ depends on countably many coordinates (for example, see \cite{en1}), there are a countable set $\Gamma(B)$ containing $B$ and corresponding to each $n$ sequences $\{\varphi_{ni}\}_{i\geq 1}\subset C(X_{\Gamma(B)},X_B)$ and $\{\phi_{ni}\}_{i\geq 1}\subset C(X_{\Gamma(B)},X_B)$
such that $\pi_{B}\circ f_{ni}=\varphi_{ni}\circ\pi_{\Gamma(B)}$ and $\pi_{B}\circ g_{ni}=\phi_{ni}\circ\pi_{\Gamma(B)}$
for all $n,i$. Then for every $f\in K$ and $n$ there exists $i_n$ such that the map $\pi_B\circ f$ is $\mathcal U_n$-close to $\pi_{B}\circ f_{ni_n}$.
Because $\pi_{B}\circ f_{ni_n}=\varphi_{ni_n}\circ\pi_{\Gamma(B)}$, we obtain that for any
$x,y\in X$ with $\pi_{\Gamma(B)}(x)=\pi_{\Gamma(B)}(y)$ we have $\pi_{B}(f(x))=\pi_{B}(f(y))$. This means that there exists a map $f_{\Gamma(B)}\in C(X_{\Gamma(B)},X_B)$ with $\pi_{B}\circ f=f_{\Gamma(B)}\circ\pi_{\Gamma(B)}$. Similarly, there exists a map $g_{\Gamma(B)}\in C(X_{\Gamma(B)},X_B)$ with
$\pi_{B}\circ f^{-1}=g_{\Gamma(B)}\circ\pi_{\Gamma(B)}$.

\begin{claim}\label{2}
For every countable set $B\subset A$ there is a countable set $\Lambda(B)\subset A$ containing $B$ such that for every $f\in K$ there exist
homeomorphisms
$f_{\Lambda(B)},g_{\Lambda(B)}\in \mathcal H(X_{\Lambda(B)})$ with $\pi_{\Lambda(B)}\circ f=f_{\Lambda(B)}\circ\pi_{\Lambda(B)}$ and
$\pi_{\Lambda(B)}\circ f^{-1}=g_{\Lambda(B)}\circ\pi_{\Lambda(B)}$.
\end{claim}
Indeed, using the notations from Claim \ref{1}, we construct an increasing sequence $B(n)$ of countable subsets of $A$ such that $B(0)=B$ and
$B(n)=\Gamma(B(n-1))$ for every $n\geq 1$. Let $\Lambda(B)=\bigcup_{n\geq 1}B(n)$. Then for every $f\in K$ there exist maps
$f_n$ and $g_n$ in $C(X_{B(n)},X_{B(n-1)})$ such that

$$\pi_{B(n-1)}\circ f=f_n\circ\pi_{B(n)}{~}\hbox{and}{~} \pi_{B(n-1)}\circ f^{-1}=g_n\circ\pi_{B(n)}.$$

The last condition implies that if $f\in K$, then for every $x,y\in X$ with $\pi_{\Lambda(B)}(x)=\pi_{\Lambda(B)}(y)$ we have
$\pi_{\Lambda(B)}(f(x))=\pi_{\Lambda(B)}(f(y))$ and $\pi_{\Lambda(B)}(f^{-1}(x))=\pi_{\Lambda(B)}(f^{-1}(y))$. Therefore, there exist homeomorphisms
$f_{\Lambda(B)},g_{\Lambda(B)}\in \mathcal H(X_{\Lambda(B)})$ satisfying the required conditions.

Now, for every $\alpha<\omega(\tau)$ take a countable set $\Lambda(\alpha)\subset A$ satisfying the hypotheses of Claim \ref{2} with  $B=\{\alpha\}$.
Let $A(\alpha)=\bigcup_{\gamma<\alpha}\Lambda(\gamma)$ if $\alpha$ is a limit ordinal, and $A(\alpha)=A(\alpha-1)\cup\Lambda(\alpha)$ otherwise. Since for every $\alpha\in A$ and $f\in K$ there exist $f_{\Lambda(\alpha)},g_{\Lambda(\alpha)}\in \mathcal H(X_{\Lambda(\alpha)})$ with $\pi_{\Lambda(\alpha)}\circ f=f_{\Lambda(\alpha)}\circ\pi_{\Lambda(\alpha)}$ and
$\pi_{\Lambda(\alpha)}\circ f^{-1}=g_{\Lambda(\alpha)}\circ\pi_{\Lambda(\alpha)}$, we have $\pi_{A(\alpha)}(f(x))=\pi_{A(\alpha)}(f(y))$ and
$\pi_{A(\alpha)}(f^{-1}(x))=\pi_{A(\alpha)}(f^{-1}(y))$ for any pair $x,y\in X$ with $\pi_{A(\alpha)}(x)=\pi_{A(\alpha)}(y)$. This  yields a homeomorphism $f_\alpha\in\mathcal H(X_{A(\alpha)})$ such that $\pi_{A(\alpha)}\circ f=f_\alpha\circ\pi_{A(\alpha)}$.
\end{proof}

The next theorem is an analogue of Mednikov's result \cite[Corollary 3]{m} stating that $\mathcal H([0,1]^A)$ is an absolute extensor for compact spaces.
\begin{thm}
Let $X=\prod_{\alpha\in A}X_\alpha$ be a product of compact metric spaces. Then $\mathcal H(X)$ is an absolute extensor
for zero-dimensional compact spaces.
\end{thm}
\begin{proof}
Suppose $Y$ is a 0-dimensional compact space and $g:P\to \mathcal H(X)$ be a map, where $P$ is closed in $Y$. If $A$ is countable, then
$\mathcal H(X)$ is a complete separable metric space and we define a set-valued map $\Phi: Y\rightsquigarrow\mathcal H(X)$,
$\Phi(y)=\{g(y)\}$ if $y\in P$ and $\Phi(y)=\mathcal H(X)$ otherwise. Since $\Phi$ is lower semi-continuous
(the set $\Phi^{-1}(U)=\{y\in Y:\Phi(y)\cap U\neq\varnothing\}$ is open in $Y$ for every open $U\subset \mathcal H(X)$), by Michael's 0-dimensional selection theorem \cite{em}, $\Phi$ admits a continuous selection $\widetilde g:Y\to\mathcal H(X)$. Obviously, $\widetilde g$ extends $g$.

Assume $A$ is an uncountable set of cardinality $\tau$. Then $A$ can be covered by a family $\xi=\{A(\alpha):\alpha\in\omega(\tau)\}$ satisfying the hypotheses of Proposition \ref{hom} with $K=g(P)$.
Then for every $\alpha\in\omega(\tau)$ and $f\in K$ we have
\begin{itemize}
\item[(*)] $\pi^{A(\alpha+1)}_{A(\alpha)}\circ f_{\alpha+1}=f_\alpha\circ\pi^{A(\alpha+1)}_{A(\alpha)}$.
\end{itemize}
Denote by $\mathcal H_\xi(X)$ the subspace of $\mathcal H(X)$ consisting of all $f$ with the following property: For every $\alpha$ there is
$f_\alpha\in\mathcal H(X_{A(\alpha)})$ such that $f_\alpha$ and $f_{\alpha+1}$ satisfy $(*)$. Since $K\subset \mathcal H_\xi(X)$, it suffice to show that $\mathcal H_\xi(X)$ is an absolute extensor for 0-dimensional compacta.

Because
$X_{A(\alpha+1)}=X_{A(\alpha)}\times X_{A(\alpha+1)\backslash A(\alpha)}$, it follows from $(*)$ that $f_{\alpha+1}$ is of the form
$f_{\alpha+1}(x,y)=(f_{\alpha}(x),g(x,y))$, where $x\in X_{A(\alpha)}$, $y\in X_{A(\alpha+1)\backslash A(\alpha)}$ and $g$ is a map
from $X_{A(\alpha+1)}$ into $X_{A(\alpha+1)\backslash A(\alpha)}$ such that for any $x\in X_{A(\alpha)}$ the map
$\varphi_g(x)$, $\varphi_g(x)(y)=g(x,y)$, is a homeomorphism of $X_{A(\alpha)\backslash A(\alpha)}$. Therefore, by \cite[Theorem 3.4.9]{en},
the correspondence $\varphi_g\rightarrow g$ is a homeomorphism between $C(X_{A(\alpha)},\mathcal H(X_{A(\alpha+1)\backslash A(\alpha)}))$ and
the subset of $C(X_{A(\alpha+1)},X_{A(\alpha+1)\backslash A(\alpha)})$ consisting of all $g$ such that for each $x\in X_{A(\alpha)}$ the map $\varphi_g(x)$ belongs to $\mathcal H(X_{A(\alpha)\backslash A(\alpha)})$. Hence, the correspondence $(f_\alpha,\varphi_g)\rightarrow f_{\alpha+1}$ provides a homeomorphism between the spaces
$\mathcal H(X_{A(\alpha)})\times C(X_{A(\alpha)},\mathcal H(X_{A(\alpha)\backslash A(\alpha)}))$ and
$\mathcal H_\alpha(X_{A(\alpha+1)})$, where $\mathcal H_\alpha(X_{A(\alpha+1)})$ consists of all homeomorphisms $f_{\alpha+1}$ on $X_{A(\alpha+1)}$ satisfying equality $(*)$. This means that there is one-to-one correspondence between $\mathcal H_\xi(X)$ and the product
\begin{itemize}
\item[(**)] $\mathcal H(X_{A(0)})\times\Pi_{\alpha<\omega(\tau)}C(X_{A(\alpha)},\mathcal H(X_{A(\alpha+1)\backslash A(\alpha)})).$
\end{itemize}
This correspondence is a homeomorphism when all function spaces carry the compact-open topology.

It remains to show that each multiple in the product $(**)$ is an absolute extensor for 0-dimensional compacta. This is true for $\mathcal H(X_{A(0)})$ because
$A(0)$ is countable. To show that each $C(X_{A(\alpha)},\mathcal H(X_{A(\alpha+1)\backslash A(\alpha)}))$ is also an absolute extensor for 0-dimensional compacta,
take a pair $L\subset Z$ of 0-dimensional compacta and a map
$$\theta: L\to C(X_{A(\alpha)},\mathcal H(X_{A(\alpha+1)\backslash A(\alpha)})).$$ Since $\mathcal H(X_{A(\alpha+1)\backslash A(\alpha)})$ is a separable complete metric space and $X_{A(\alpha)}$ is a compactum, $C(X_{A(\alpha)},\mathcal H(X_{A(\alpha+1)\backslash A(\alpha)}))$ is a complete metric space. Then, as above, we can apply Michael's 0-dimensional selection theorem to find an extension
$\widetilde\theta:Z\to C(X_{A(\alpha)},\mathcal H(X_{A(\alpha+1)\backslash A(\alpha)}))$ of $\theta$.
\end{proof}

Everywhere below by $\mathfrak{C}$ we denote the Cantor set.
\begin{cor}
Let $P$ be a proper closed subset of $\mathfrak{C}^A$ and $f$ be an autohomeomorphism of $P$. Suppose there exist a proper subset $B\subset A$ and an autohomeomorphism $f_B$ of $P_B$  such that
\begin{itemize}
\item $P=P_{B}\times\mathfrak{C}^{A\backslash B}$;
\item $f_B\circ\pi_B=\pi_B\circ f$;
\item $f_B$ can be extended to a homeomorphism $\widetilde f_B\in\mathcal H(\mathfrak{C}^B)$.
\end{itemize}
Then $f$ can be extended to a homeomorphism $\widetilde f\in\mathcal H(\mathfrak{C}^A)$ such that $\widetilde f_B\circ\pi_B=\pi_B\circ\widetilde f$.
\end{cor}
\begin{proof}
Since $f_B\circ\pi_B=\pi_B\circ f$, $f$ is of the form $f(x,y)=(f_B(x),h(x,y))$ with $(x,y)\in P_{B}\times\mathfrak{C}^{A\backslash B}$ such that for each $x\in P_{B}$ the map $\varphi_x$, defined by $\varphi_x(y)=h(x,y)$, belongs to
$\mathcal H(\mathfrak{C}^{A\backslash B})$. So, we have a map $\varphi:P_{B}\to\mathcal H(\mathfrak{C}^{A\backslash B})$, see \cite[Theorem 3.4.9]{en}. By Theorem 3.4, we can extend $\varphi$ to a map $\widetilde\varphi: \mathfrak{C}^B\to\mathcal H(\mathfrak{C}^{A\backslash B})$ and define $\widetilde h:\mathfrak{C}^A\to \mathfrak{C}^{A\backslash B}$,
$\widetilde h(x,y)=\widetilde\varphi(x)(y)$, where $(x,y)\in \mathfrak{C}^B\times \mathfrak{C}^{A\backslash B}$. Finally, $\widetilde f(x,y)=(\widetilde f_B,\widetilde h(x,y))$ provides a homeomorphism in $\mathcal H(\mathfrak{C}^A)$ extending $f$ with $\widetilde f_B\circ\pi_B=\pi_B\circ\widetilde f$.
\end{proof}

\section{Proof of Theorem 1.1}

The next lemma with a little bit different formulation was proved in \cite[Lemma 3.1]{sv}.
\begin{lem}
Let $X, Y$ be 0-dimensional paracompact spaces, $\mathfrak{C}$ be the Cantor set and $P'\subset X\times\mathfrak{C}$, $K'\subset Y\times\mathfrak{C}$ be closed sets such that $\pi_X(P')=X$ and $\pi_{Y}(K')=Y$. Suppose $f:P'\to K'$ and $g:X\to Y$ are homeomorphisms with $g\circ\pi_X=\pi_{Y}\circ f$, and there are proper closed sets $F_X\subset X$ and  $F_Y\subset Y$ such that:
\begin{itemize}
\item[(i)] $g(F_X)=F_Y$;
\item[(ii)] $F_X\times\mathfrak{C}\subset P'$ and $F_Y\times\mathfrak{C}\subset K'$;
\item[(iii)] $\pi_{\mathfrak{C}}((\{x\}\times\mathfrak{C})\cap P')$ and $\pi_{\mathfrak{C}}((\{y\}\times\mathfrak{C})\cap K')$
are nowhere dense in $\mathfrak{C}$ for all $x\in X\backslash F_X$ and $y\in Y\backslash F_Y$.
\end{itemize}
Then $f$ can be extended to a homeomorphism
$\widetilde f:X\times\mathfrak{C}\to Y\times\mathfrak{C}$ such that $g\circ\pi_X=\pi_{Y}\circ\widetilde f$.
\end{lem}

Suppose $g:X\to Y$ and $f:X\times Z\to Y\times Z$ are two homeomorphisms and let $\pi_X:X\times Z\to X$, $\pi_Y:Y\times Z\to Y$ be the corresponding projections. We say that $f$ is a {\em fiberwise homeomorphism with respect to $g$} if  $\pi_Y\circ f=g\circ\pi_X$.

Everywhere below we suppose that $X$ is an uncountable power of the Cantor set, $P\subset X$ is closed and $f\in\mathcal H(P)$
with $f(P^{(\lambda)})=P^{(\lambda)}$ for every cardinal $\lambda$.
We follow the notations and the definitions from Section 2.
\begin{lem}
Let $X=\mathfrak{C}^A$ and $\Gamma\subset\Lambda$ be $f$-admissible subsets of $A$  with $\Lambda\backslash\Gamma$  countable. Suppose there exists a closed set $F\subset P_\Gamma$ such that %$f_\Gamma(F)=F'$ and
$F\times X_{\Lambda\backslash\Gamma}\subset P_\Lambda$% $F'\times X_{\Lambda\backslash\Gamma}\subset P_\Lambda$ and
and $f_\Lambda(F\times X_{\Lambda\backslash\Gamma})=f_\Gamma(F)\times X_{\Lambda\backslash\Gamma}$.
Then there exists a closed $G_\delta$-set $L$ in $P_\Gamma$ containing $F$ and a fiberwise homeomorphism
$f_\Lambda':L\times X_{\Lambda\backslash\Gamma}\to f_\Gamma(L)\times X_{\Lambda\backslash\Gamma}$ with respect to $f_\Gamma|L$
extending the homeomorphism
$f_\Lambda|((L\times X_{\Lambda\backslash\Gamma})\cap P_\Lambda)$.
\end{lem}
\begin{proof}
The case $\Gamma=\Lambda$ is trivial, so let $\Lambda\backslash\Gamma\neq\varnothing$.
Since  $\pi^\Lambda_\Gamma\circ f_\Lambda=f_\Gamma\circ\pi^\Lambda_\Gamma$,
there is a map $\phi:F\times X_{\Lambda\backslash\Gamma}\to X_{\Lambda\backslash\Gamma}$ such that $f_\Lambda(x,y)=(f_\Gamma(x),\phi(x,y))$ for all $(x,y)\in F\times X_{\Lambda\backslash\Gamma}$
and the equality  $\varphi_x(y)=\phi(x,y)$ defines a homeomorphism of $X_{\Lambda\backslash\Gamma}$ for every $x\in F$. Therefore, we have a continuous map $\varphi:F\to\mathcal H(X_{\Lambda\backslash\Gamma})$, $\varphi(x)=\varphi_x$. Then by Theorem 3.4, $\varphi$ can be extended to a map
$\widetilde\varphi:P_\Gamma\to\mathcal H(X_{\Lambda\backslash\Gamma})$. Fix a countable, finitely additive base $\mathcal B$ of $X_{\Lambda\backslash\Gamma}$ consisting of clopen sets and let
$\Omega$ be the family of all finite, disjoint clopen covers of $X_{\Lambda\backslash\Gamma}$ whose elements belong to $\mathcal B$.
Because $\mathcal B$ is countable and finitely additive, $\Omega$ is countable and contains every disjoint clopen cover of $X_{\Lambda\backslash\Gamma}$.
Moreover, for every $\omega=\{U_1,U_2,..,U_k\}\in\Omega$ and $x\in F$
the family
$\{\{x\}\times U_1,\{x\}\times U_2,..,\{x\}\times U_k\}$ is a clopen disjoint cover of $\{x\}\times X_{\Lambda\backslash\Gamma}$
and the set
$$O_\omega(\varphi_x)=\{h\in\mathcal H(X_{\Lambda\backslash\Gamma}):h(U_i)\subset V_i, i=1,..,k\},$$ is a neighborhood of $\varphi_x$ in $\mathcal H(X_{\Lambda\backslash\Gamma})$, where $V_i=\varphi_x(U_i)$.
We have $f_\Lambda(\{x\}\times U_i)=f_\Gamma(x)\times V_i$, $i=1,2,..,k$. Therefore, for every $\omega\in\Omega$ and $x\in F$ there exists a neighborhood $O_\omega(x)$ of $x$
in $P_\Gamma$ such that $f_\Lambda((\{z\}\times U)\cap P_\Lambda)\subset \{f_\Gamma(z)\}\times\varphi_x(U)$ for all $U\in\omega$ and $z\in O_\omega(x)$.
\begin{claim}
If $h\in O_\omega(\varphi_x)$, then $h(U)=\varphi_x(U)$ for every $U\in\omega$.
\end{claim}
Obviously, $h(U)\subset\varphi_x(U)$ for every $U\in\omega$. So, the claim follows from the fact that $\omega$ is a disjoint cover of $X_{\Lambda\backslash\Gamma}$ and $h, \varphi_x$ are homeomorphisms of $X_{\Lambda\backslash\Gamma}$.

Then $G_\omega(x)=\widetilde\varphi^{-1}(O_\omega(\varphi_x))\cap O_\omega(x)$ is a neighborhood of $x$ in $P_\Gamma$. Let $W_\omega$ be a clopen neighborhood of $F$ in $P_\Gamma$ with
$W_\omega\subset\bigcup_{x\in F}G_\omega(x)$. Finally, let $L=\bigcap_{\omega\in\Omega}W_\omega$. Then for every $z\in L$
the map
$\widetilde\varphi_z=\widetilde\varphi(z)$ is a homeomorphism of  $X_{\Lambda\backslash\Gamma}$ such that
$\widetilde\varphi_z(U)\subset\varphi_{x}(U)$ for each $U\in\omega$ provided $z\in \widetilde\varphi^{-1}(O_\omega(\varphi_x))$. Now we define the homeomorphism
$f_\Lambda':L\times X_{\Lambda\backslash\Gamma}\to f_\Gamma(L)\times X_{\Lambda\backslash\Gamma}$, $f_\Lambda'(z,y)=(f_\Gamma(z),\widetilde\varphi_z(y))$. Let show that $f_\Lambda'$ extends
$f_\Lambda|((L\times X_{\Lambda\backslash\Gamma})\cap P_\Lambda)$.
If $(z,y)\in (L\times X_{\Lambda\backslash\Gamma})\cap P_\Lambda$, then for every $\omega\in\Omega$ there is $x_\omega\in F$ and $U_\omega\in\omega$ such that $z\in G_\omega(x_\omega)$ and $y\in U_\omega$.
Consequently, $\widetilde\varphi_z\in O_\omega(\varphi_{x_\omega})$ which, by Claim 4.3, implies $\widetilde\varphi_{z}(U_\omega)=\varphi_{x_\omega}(U_\omega)$.  %$\widetilde\varphi_{z}(U_\omega)\subset\varphi_{x_\omega}(U_\omega)$.
Hence, $f_\Lambda'(z,y)\in\{f_\Gamma(z)\}\times\widetilde\varphi_{z}(U_\omega)=\{f_\Gamma(z)\}\times\varphi_{x_\omega}(U_\omega)$.
%\subset\{f_\Gamma(z)\}\times\varphi_{x_\omega}(U_\omega).$
On the other hand, $z\in O_\omega(x_\omega)$ yields
$f_\Lambda(z,y)\in\{f_\Gamma(z)\}\times\varphi_{x_\omega}(U_\omega)$.
%Because $z\in\widetilde\varphi^{-1}(O_\omega(\varphi_{x_\omega}))$,
%$\varphi_{x_\omega}(U_\omega)=\widetilde\varphi_z(U_\omega)$.
Therefore,
%$$f_\Lambda'(z,y)\in\{f_\Gamma(z)\}\times\widetilde\varphi_{z}(U_\omega){~}\hbox{and}{~}
$$f_\Lambda'(z,y), f_\Lambda(z,y)\in\{f_\Gamma(z)\}\times\widetilde\varphi_{z}(U_\omega)$$
for every $\omega\in\Omega$, where $U_\omega$ is the only element from $\omega$ containing $y$. This means that
$f_\Lambda'(z,y)=f_\Lambda(z,y)$. Indeed, otherwise there would be $\omega'\in\Omega$ and two different elements $U_i',U_j'\in\omega'$ with $\widetilde\varphi_z(y)\in U_i'$ and $\pi_{\Lambda\backslash\Gamma}(f_\Lambda(z,y))\in U_j'$. Then $\omega=\widetilde\varphi_z^{-1}(\omega')\in\Omega$ and
$U_\omega=\widetilde\varphi_z^{-1}(U_i')$, but $f_\Lambda(z,y)\not\in\{f_\Gamma(z)\}\times\widetilde\varphi_z(U_\omega)$
while $f_\Lambda'(z,y)\in\{f_\Gamma(z)\}\times\widetilde\varphi_z(U_\omega)$.
\end{proof}

\begin{lem}
Let the saturated sets $C\subset B$  satisfy the hypotheses of Proposition $2.3$ with $X=\mathfrak{C}^A$. Then for every saturated set $\Gamma\subset B$  of cardinality $\leq\lambda$ containing $C$ there exists a saturated set $\Lambda(\Gamma)\subset B$ containing $\Gamma$ such that $\Lambda(\Gamma)\backslash\Gamma\neq\varnothing$ is countable and the homeomorphism $f_{\Lambda(\Gamma)}$ can be extended to a fiberwise homeomorphism
$\widetilde f_{\Lambda(\Gamma)}:P_\Gamma\times X_{\Lambda(\Gamma)\backslash\Gamma}\to P_\Gamma\times X_{\Lambda(\Gamma)\backslash\Gamma}$ with respect to $f_\Gamma$.
\end{lem}
\begin{proof}
Recall that for every $f$-admissible set $T\subset A$ we denote by $P_T^{(\lambda)}$ the set $\pi_T(P^{(\lambda)})$.
Since $B$ and $\Gamma$ are saturated and $B$ is of cardinality $\lambda^+$, there is a saturated set $\Lambda(1)\subset B$ containing $\Gamma$ such that $\Lambda(1)\backslash\Gamma\neq\varnothing$ is countable.
 Observe that $\pi_\Gamma^{-1}(P_\Gamma^{(\lambda)})=P^{(\lambda)}$ and $\pi_{\Lambda(1)}^{-1}(P_{\Lambda(1)}^{(\lambda)})=P^{(\lambda)}$
 because $\pi_C^{-1}(P_C^{(\lambda)})=P^{(\lambda)}$. The last two equalities together with $f(P^{(\lambda)})=P^{(\lambda)}$ imply that
$P_{\Lambda(1)}^{(\lambda)}=P_{\Gamma}^{(\lambda)}\times X_{\Lambda(1)\backslash\Gamma}$, $f_\Gamma(P_{\Gamma}^{(\lambda)})=P_{\Gamma}^{(\lambda)}$ and $f_{\Lambda(1)}(P_{\Lambda(1)}^{(\lambda)})=P_{\Lambda(1)}^{(\lambda)}$.

 Then, we can apply Lemma 4.2 for the pair $\Gamma\subset \Lambda(1)$ to obtain a closed $G_\delta$-set $L_1$ in $P_\Gamma$ containing $P_\Gamma^{(\lambda)}$ and a fiberwise homeomorphism $f_1:L_1\times X_{\Lambda(1)\backslash\Gamma}\to
f_\Gamma(L_0)\times X_{\Lambda(1)\backslash\Gamma}$ with respect to $f_\Gamma|L_1$ extending $f_{\Lambda(1)}|((L_1\times X_{\Lambda(1)\backslash\Gamma})\cap P_{\Lambda(1)})$.
%Let $L_1=L\times X_{\Lambda(1)\backslash\Gamma}$ and $L_1'=f_\Gamma(L)\times X_{\Lambda(1)\backslash\Gamma}$.
Since the sets $F_1=P_{\Gamma}\backslash L_1$ and $F_1'=P_{\Gamma}\backslash f_\Gamma(L_1)$ are $\sigma$-compact disjoint from $P_\Gamma^{(\lambda)}$,
according to Lemma 2.5 there is a saturated set
$\Lambda(2)\subset B$ containing $\Lambda(1)$ such that $\Lambda(2)\backslash\Lambda(1)\neq\varnothing$ is countable and
the sets
$\pi^{\Lambda(2)}_{\Lambda(2)\backslash\Gamma}((\pi_{\Gamma}^{\Lambda(2)})^{-1}(x)\cap P_{\Lambda(2)})$  and
$\pi^{\Lambda(2)}_{\Lambda(2)\backslash\Gamma}((\pi_{\Gamma}^{\Lambda(2)})^{-1}(y)\cap P_{\Lambda(2)})$
are nowhere dense in $X_{\Lambda(2)\backslash\Gamma}$ for all $x\in F_1$ and $y\in F'_1$.

Since
$P_{\Lambda(2)}^{(\lambda)})=P_{\Lambda(1)}^{(\lambda)}\times X_{\Lambda(2)\backslash\Lambda(1)}$ and
$f_{\Lambda(2)}(P_{\Lambda(2)}^{(\lambda)})=P_{\Lambda(2)}^{(\lambda)}$,
we can apply Lemma 4.2 (with $F=P_{\Lambda(1)}^{(\lambda)}$) for the saturated sets $\Lambda(1)\subset\Lambda(2)$.
Therefore, there exist
a closed $G_\delta$-set $L_2'$ in $P_{\Lambda(1)}$ containing $P_{\Lambda(1)}^{(\lambda)})$ and a fiberwise homeomorphism
$f_2':L_2'\times X_{\Lambda(2)\backslash\Lambda(1)}\to
f_{\Lambda(1)}(L_2')\times X_{\Lambda(2)\backslash\Lambda(1)}$ with respect to $f_{\Lambda(1)}|L_2'$ extending $f_{\Lambda(2)}|((L_2'\times X_{\Lambda(2)\backslash\Lambda(1)})\cap P_{\Lambda(2)})$. Since $L_1\times X_{\Lambda(1)\backslash\Gamma}$ is a $G_\delta$-set in $P_\Gamma\times X_{\Lambda(1)\backslash\Gamma}$
and $(\pi^{\Lambda(1)}_\Gamma)^{-1}(P_\Gamma^{(\lambda)})=P_{\Lambda(1)}^{(\lambda)}$,
there exists a closed $G_\delta$-set $L_2$ in $P_\Gamma$ containing $P_\Gamma^{(\lambda)}$ such that $L_2\subset L_1$ and
$L_2\times X_{\Lambda(1)\backslash\Gamma}\subset L_2'$. So,
$L_2\times X_{\Lambda(2)\backslash\Gamma}\subset L_2'\times X_{\Lambda(2)\backslash\Lambda(1)}$.
Let $f_2$ be the restriction of $f_2'$ on $L_2\times X_{\Lambda(2)\backslash\Gamma}$. Then $f_2$ is a fiberwise homeomorphism between
$L_2\times X_{\Lambda(2)\backslash\Gamma}$ and $f_\Gamma(L_2)\times X_{\Lambda(2)\backslash\Gamma}$ with respect to $f_{\Lambda(1)}|(L_2\times X_{\Lambda(1)\backslash\Gamma})$ extending
$f_{\Lambda(2)}|((L_2\times X_{\Lambda(2)\backslash\Gamma})\cap P_{\Lambda(2)})$.

In this way we can construct an increasing sequence $\{\Lambda(n)\}$ of saturated subsets of $B$ each containing $\Gamma$, a decreasing sequence $\{L_n\}$ of closed $G_\delta$-sets in $P_\Gamma$ each containing $P_\Gamma^{(\lambda)}$, and homeomorphisms $f_n$ between $L_n\times X_{\Lambda(n)\backslash\Gamma}$ and
$f_\Gamma(L_n)\times X_{\Lambda(n)\backslash\Gamma}$ extending $f_{\Lambda(n)}|((L_n\times X_{\Lambda(n)\backslash\Gamma})\cap P_{\Lambda(n)})$
such that for each $n$ we have:
\begin{itemize}
\item[(a)] $\Lambda(n+1)\backslash\Lambda(n)\neq\varnothing$ is countable;
\item[(b)] The set $\pi^{\Lambda(n+1)}_{\Lambda(n+1)\backslash\Gamma}((\pi_{\Gamma}^{\Lambda(n+1)})^{-1}(x)\cap P_{\Lambda(n+1)})$ is nowhere dense in $X_{\Lambda(n+1)\backslash\Gamma}$ for all $x\in F_n=P_{\Gamma}\backslash L_n$;
\item[(c)] The set $\pi^{\Lambda(n+1)}_{\Lambda(n+1)\backslash\Gamma}((\pi_{\Gamma}^{\Lambda(n+1)})^{-1}(y)\cap P_{\Lambda(n+1)})$
is nowhere dense in $X_{\Lambda(n+1)\backslash\Gamma}$ for all $y\in F'_n=P_{\Gamma}\backslash f_\Gamma(L_n)$;
\item[(d)] Each $f_{n+1}$ is a fiberwise homeomorphism with respect to the restriction $f_n|(L_{n+1}\times X_{\Lambda(n)\backslash\Gamma})$.
\end{itemize}

Let $\Lambda(\Gamma)=\bigcup_{n\geq 0}\Lambda(n)$ and $L=\bigcap_{n\geq 0}L(n)$. For every $n$ consider the map $p_n: L\times X_{\Lambda(n+1)\backslash\Gamma}\to L\times X_{\Lambda(n)\backslash\Gamma}$ defined by $p_n(x,y)=(x,\pi^{\Lambda(n+1)}_{\Lambda(n)}(y))$.
Then $L\times X_{\Lambda(\Gamma)\backslash\Gamma}$ is the limit of the inverse  sequence $\{L\times X_{\Lambda(n)\backslash\Gamma},p_n\}$ and
$f_\Gamma(L)\times X_{\Lambda(\Gamma)\backslash\Gamma}$ is the limit of the inverse  sequence $\{f_\Gamma(L)\times X_{\Lambda(n)\backslash\Gamma},p_n\}$.
Since $p_n\circ f_{n+1}=f_n\circ p_n$ for every $n$,
the homeomorphisms $f_n$ provide a homeomorphism $f_\infty$ between $L\times X_{\Lambda(\Gamma)\backslash\Gamma}$ and $f_\Gamma(L)\times X_{\Lambda(\Gamma)\backslash\Gamma}$ extending $f_{\Lambda(\Gamma)}|((L\times X_{\Lambda(\Gamma)\backslash\Gamma})\cap P_{\Lambda(\Gamma)})$.
 Moreover, items (b) and (c) imply that the sets
 $\pi^{\Lambda(\Gamma)}_{\Lambda(\Gamma)\backslash\Gamma}((\pi_\Gamma^{\Lambda(\Gamma)})^{-1}(x)\cap P_{\Lambda(\Gamma)})$ and
$\pi^{\Lambda(\Gamma)}_{\Lambda(\Gamma)\backslash\Gamma}((\pi_{\Gamma}^{\Lambda(\Gamma)})^{-1}(y)\cap P_{\Lambda(\Gamma)})$ are nowhere dense in $X_{\Lambda(\Gamma)\backslash\Gamma}$ for all $x\in P_{\Gamma}\backslash L$ and $y\in P_{\Gamma}\backslash f_\Gamma(L)$. Finally, by Lemma 4.1, $f_\infty$ can be extended to a homeomorphism $\widetilde f_{\Lambda(\Gamma)}$ between $P_\Gamma\times X_{\Lambda(\Gamma)\backslash\Gamma}$ and $P_\Gamma\times X_{\Lambda(\Gamma)\backslash\Gamma}$.
Obviously, $\widetilde f_{\Lambda(\Gamma)}$ is a fiberwise extension of $f_{\Lambda(\Gamma)}$ with respect to $f_\Gamma$.
\end{proof}

\begin{cor}
Let the saturated sets $C\subset B$  satisfy the hypotheses of Proposition $2.3$
with  $X=\mathfrak{C}^A$.
Suppose also that $f_C$
can be extended to a homeomorphism $\widetilde f_C$ of $X_C$.
Then the homeomorphism $f_B$ can be extended to a homeomorphism
$\widetilde f_B$ of $X_B$ such that $\pi^B_C\circ\widetilde f_B=\widetilde f_C\circ\pi^B_C$.
\end{cor}

\begin{proof}
Using Lemma 4.4 (and the notations from that lemma), we can represent $B$ as an increasing family $\{B(\alpha):\alpha<\omega(\lambda^+)\}$ of saturated sets $B(\alpha)$ each of cardinality $\leq\lambda$ such that:
\begin{itemize}
\item[(3)] $B(0)=C$;
\item[(4)] $B(\alpha)=\bigcup_{\beta<\alpha}B(\beta)$ if $\alpha$ is a limit ordinal;
\item[(5)]  $B(\alpha+1)=\Lambda(B(\alpha))$, so $B(\alpha+1)\backslash B(\alpha)$ is countable for all $\alpha$;
\item[(6)] $f_{B(\alpha+1)}$ can be extended to a fiberwise homeomorphism\\ $f_{B(\alpha+1)}':P_{B(\alpha)}\times X_{B(\alpha+1)\backslash B(\alpha)}\to P_{B(\alpha)}\times X_{B(\alpha+1)\backslash B(\alpha)}$  with respect to $f_{B(\alpha)}$.
\end{itemize}
We are going to prove that each $f_{B(\alpha)}$ can be extended to a homeomorphism $\widetilde f_{B(\alpha)}$ of $X_{B(\alpha)}$ satisfying the following equality
\begin{itemize}
\item[(7)] $\displaystyle\pi_{B(\alpha)}^{B(\alpha+1)}\circ\widetilde f_{B(\alpha+1)}=\widetilde
f_{B(\alpha)}\circ\pi_{B(\alpha)}^{B(\alpha+1)}$.
\end{itemize}
The proof is by transfinite induction. The extension $\widetilde f_{B(0)}$ exists according to our assumption since $B(0)=C$. If $\widetilde f_{B(\beta)}$ is defined for each $\beta<\alpha$, where $\alpha$ is a limit ordinal, then item $(4)$
implies the existence of $\widetilde f_{B(\alpha)}$. Therefore, we need only to define $\widetilde f_{B(\alpha+1)}$ provided $\widetilde f_{B(\alpha)}$
exists. To that end, since $B(\alpha+1)=\Lambda(B(\alpha))$, according to Lemma 4.4, $f_{B(\alpha+1)}$ can be extended to a
fiberwise homeomorphism $f_{B(\alpha+1)}'$ of $P_{B(\alpha)}\times X_{B(\alpha+1)\backslash B(\alpha)}$ with respect to $f_{B(\alpha)}$. Then, by Corollary 3.5, $f_{B(\alpha+1)}'$ is extended to a homeomorphism $\widetilde f_{B(\alpha+1)}$ satisfying condition $(7)$.
\end{proof}

{\em Proof of Theorem $1.1$.}
Let show first that the proof is reduced to the case of one subset $P\subset D^A$ and an autohomeomorphism $f\in\mathcal H(P)$. Indeed, take two disjoint copies $X$ and $Y$ of $D^A$ with $P\subset X$ and $K\subset Y$, and let
$Q=P\biguplus K$ be the disjoint union of $P$ and $K$. Obviously, $X\biguplus Y$ is homeomorphic to $D^A$ %$Q$ is negligible in $X\biguplus Y$
and $g=f\biguplus f^{-1}$ is an autohomeomorphism of $Q$ with $g(Q^{(\lambda)})=Q^{(\lambda)}$ for all cardinals $\lambda$. Suppose $f\biguplus f^{-1}$ can be extended to a homeomorphism $F:X\biguplus Y\to X\biguplus Y$.
Choose two clopen neighborhoods $X'$ and $Y'$ of $P$ and $K$ in $X$ and $Y$, respectively, with $X\backslash X'\neq\varnothing\neq Y\backslash Y'$
such that $F(X')=Y'$.
Then %$X'=F^{-1}(Y)\cap X$ and $Y'=F(X')$ are clopen neighborhoods of $P$ and $K$ in $X$ and $Y$, respectively. Consider the clopen subsets
%$X\backslash X'$ and $Y\backslash Y'$ of $D^A$. Both of them are either empty or non-empty. In case $X\backslash X'=Y\backslash Y'=\varnothing$, %$F$ is a homeomorphism between $X$ and $Y$ extending $f$. If $X\backslash X'\neq\varnothing\neq Y\backslash Y'$,
there is a homeomorphism
$G:X\backslash X'\to Y\backslash Y'$, and $F|X'$ and $G$ provide a homeomorphism $\widetilde f:X\to Y$ extending $f$.
Therefore, we can suppose that we have one subset $P$ of $D^A$ and a homeomorphism $f\in\mathcal H(P)$. Because $f$ preserves the interior of $P$, we can also assume $P$ is nowhere dense in $D^A$. Moreover, we identify $D^A$ with $X=\mathfrak{C}^A$, where $A$ is an uncountable set of cardinality $\tau$. Recall that $\mathfrak{L}=\mathfrak{L}_P$ denotes the set of all infinite cardinal numbers $\mu\leq\tau$ with $P^{(\mu)}\neq\varnothing$, and let $\mathfrak{L}'=\mathfrak{L}\cup\{\aleph_0\}$. Obviously, if $\lambda\in\mathfrak{L}$, then  $\mu\in\mathfrak{L}$ for all $\mu\geq\lambda$, in particular, $\tau\in\mathfrak{L}$.
Take a functionally open and dense subset $U$ of $X\backslash P$ and a countable saturated set $C\subset A$ with
$\pi_C^{-1}(\pi_C(U))=U$ (that is possible because every continuous function on $X$ depends on countably many coordinates). Then $P_C$ is nowhere dense in $X_C$. If $\aleph_0\not\in\mathfrak{L}$, we put $\Gamma(\aleph_0)=C$. In case $\aleph_0\in\mathfrak{L}$, by Proposition 2.1 and \cite{ef}, there exits a countable saturated set $C_1\subset A$ with $\pi_{C_1}^{-1}(\pi_{C_1}(P^{(\aleph_0)}))=P^{(\aleph_0)}$, and  we denote $\Gamma(0)=C\cup C_1$. Therefore, in both cases $P_{\Gamma(\aleph_0)}$ is nowhere dense in $X_{\Gamma(0)}$, $\Gamma(0)$ is a countable saturated set and $\pi_{\Gamma(0)}^{-1}(P_{\Gamma(0)}^{(\aleph_0)})=P^{(\aleph_0)}$.

The case when $P^{(\lambda)}=\varnothing$ for all $\lambda<\tau$ was settled in \cite{sv}. Hence, we suppose that $\mathfrak{L}$ is not a single-point set. Then
there are two passible cases: either $P\neq P^{(\lambda)}$ for all $\lambda<\tau$ or $P=P^{(\lambda)}$ for some $\lambda<\tau$.

Consider first the case $P\neq P^{(\lambda)}$ for all $\lambda<\tau$. 
Denote by $\mathfrak{L}_c$ the cardinals $\lambda$ at which $\mathfrak{L}$ is continuous.
We construct a family $\{\Gamma(\lambda):\lambda\leq\tau\}$ of saturated subsets of $A$ such that:
\begin{itemize}
\item[(8)] $\Gamma(\lambda)$ is of cardinality $\lambda$ and $\Gamma(\aleph_0)$ is the set constructed above;
\item[(9)] $\Gamma(\lambda)\subset \Gamma(\mu)$ provided $\lambda<\mu$, and $\Gamma(\tau)=A$;
\item[(10)] Every compact set $F\subset P_{\Gamma(\lambda^+)}\backslash P_{\Gamma(\lambda^+)}^{(\lambda)}$ is $\lambda^+$-negligible in $X_{\Gamma(\lambda^+)}$;
\item[(11)] $\pi_{\Gamma(\lambda)}^{-1}(P_{\Gamma(\lambda)}^{(\lambda)})=P^{(\lambda)}$;
\item[(12)]  $\Gamma(\lambda)=\bigcup_{\gamma<\lambda}\Gamma(\gamma)$ if $\lambda\in\mathfrak{L}_c$.
\end{itemize}
The construction is by transfinite induction.
Suppose the sets $\Gamma(\lambda)$ are already constructed for all $\lambda<\lambda_0$. If $\lambda_0=\lambda^+$ for some $\lambda<\lambda_0$, by Proposition 2.3, there exists a saturated set $\Gamma(\lambda_0)$ of cardinality $\lambda_0$ containing $\Gamma(\lambda)$ such that
$\pi_{\Gamma(\lambda_0)}^{-1}(P_{\Gamma(\lambda_0)}^{(\lambda_0)})=P^{(\lambda_0)}$ and every compact set in 
$P_{\Gamma(\lambda_0)}\backslash P_{\Gamma(\lambda_0)}^{(\lambda)}$ is $\lambda_0$-negligible in $X_{\Gamma(\lambda_0)}$. If $\lambda_0$ is a limit cardinal such that $\mathfrak{L}$ is continuous at $\lambda_0$, let $\Gamma(\lambda_0)$ be the union of all $\Gamma(\lambda)$, $\lambda<\lambda_0$. In this case, since each $\Gamma(\gamma)$, $\gamma<\lambda_0$, satisfies condition $(11)$ and $\lambda_0\in\mathfrak{L}_c$,
$\lambda_0$ also satisfies $(11)$. 
Finally, if $\mathfrak{L}$ is discontinuous at $\lambda_0$, we take a saturated set $\Gamma(\lambda_0)$ of cardinality $\lambda_0$ containing $\bigcup_{\lambda<\lambda_0}\Gamma(\lambda)$ with $\pi_{\Gamma(\lambda_0)}^{-1}(P_{\Gamma(\lambda_0)}^{(\lambda_0)})=P^{(\lambda_0)}$.

Denote by $f_\lambda$ the homeomorphism $f_{\Gamma(\lambda)}$. We are going to extend by transfinite induction each $f_\lambda$ to a homeomorphism
$\widetilde{f_\lambda}\in H(X_{\Gamma(\lambda)})$ with
\begin{itemize}
\item [(***)] $\displaystyle\pi_{\Gamma(\gamma)}^{\Gamma(\lambda)}\circ\widetilde f_{\lambda}=\widetilde
f_{\gamma}\circ\pi_{\Gamma(\gamma)}^{\Gamma(\lambda)}$ for all $\gamma<\lambda$.
\end{itemize}
 The homeomorphism $\widetilde f_{0}$ exists by
Knaster-Reichbach's theorem mentioned above because $P_{\Gamma(0)}$ is nowhere dense in $X_{\Gamma(0)}$. Suppose
$\widetilde{f_\gamma}$ is already constructed for all $\gamma<\lambda$. If $\lambda$ has a predecessor $\gamma_0$, i.e., $\lambda=\gamma_0^+$,    then the sets $\Gamma(\gamma_0)\subset\Gamma(\lambda)$
satisfy the hypotheses of Proposition 2.3 and, by Corollary 4.5, $f_{\lambda}$ can be extended to a homeomorphism $\widetilde f_{\lambda}\in H(X_{\Gamma(\lambda)})$ such that condition $(***)$ holds. 
If $\lambda$ is a limit cardinal then either $\lambda\in\mathfrak{L}_c$ or $\lambda\in\mathfrak{L}_d$. In case $\lambda\in\mathfrak{L}_c$, we have $\Gamma(\lambda)=\bigcup_{\gamma<\lambda}\Gamma(\gamma)$. So, we can define $\widetilde f_{\lambda}$ as the limit of all $\widetilde f_{\gamma}$, $\gamma<\lambda$. Suppose $\lambda\in\mathfrak{L}_d$, so $P^{(\lambda)}\neq\varnothing$. 
Since, $\mathfrak{L}_d$ is discrete in $\mathfrak{L}$, there is $\gamma_0<\lambda$ such that $(\gamma_0,\lambda)\subset \mathfrak{L}_c$.
In this case the set $\Gamma(\lambda)$ may not be the union of all $\Gamma(\gamma)$ with $\gamma<\lambda$, and we cannot define $\widetilde f_{\lambda}$ to be the limit of the homeomorphisms
$\widetilde f_{\gamma}$, $\gamma<\lambda$. But we are going to show that there is an extension $\widetilde f_\lambda$ of $f_\lambda$ such that   
$\displaystyle\pi_{\Gamma(\gamma_0)}^{\Gamma(\lambda)}\circ\widetilde f_{\lambda}=\widetilde
f_{\gamma_0}\circ\pi_{\Gamma(\gamma_0)}^{\Gamma(\lambda)}$.
\begin{claim}
 There is an increasing family of saturated sets $\{\Lambda(\gamma)\}_{\gamma_0\leq\gamma<\lambda}$ each of cardinality $\gamma$ such that:
\begin{itemize}  
\item[(13)] $\Gamma(\lambda)=\bigcup_{\gamma_0\leq\gamma<\lambda}\Lambda(\gamma)$ and $\Lambda(\gamma_0)=\Gamma(\gamma_0)$; 
\item[(14)] $\Lambda(\beta)=\bigcup_{\gamma<\beta}\Lambda(\gamma)$ for all limit cardinals $\beta\in (\gamma_0,\lambda)$;
\item[(15)] $\pi_{\Lambda(\gamma)}^{-1}(P_{\Lambda(\gamma)}^{(\gamma)})=P^{(\gamma)}$;
\item[(16)] Each pair $\Lambda(\gamma)\subset\Lambda(\gamma^+)$, $\gamma_0\leq\gamma<\lambda$, satisfies the hypotheses  of Proposition 2.3.
\end{itemize}
\end{claim}
Indeed, since $\Gamma(\lambda)$ and $\Lambda(\gamma_0)$ are saturated sets of cardinality $\lambda$ and $\gamma_0$, respectively, we can assume that $\Gamma(\lambda)=\bigcup_{\alpha\in\Gamma(\lambda)}B(\alpha)$ and  $\Lambda(\gamma_0)=\bigcup_{\alpha\in\Lambda(\gamma_0)}B(\alpha)$, where all  $B(\alpha)$ are countable saturated sets. We also represent $\Gamma(\lambda)\backslash\Gamma(\gamma_0)$ as the union of an increasing
family $\{\Theta(\gamma):\gamma_0\leq\gamma<\lambda\}$ such that  each $\Theta(\gamma)$ has cardinality $\gamma$. Since $\Gamma(\gamma_0)$ and $\Gamma(\lambda)$ satisfy condition $(11)$, there is set $\Theta\subset\Gamma(\lambda)$ of cardinality $\gamma_0^+$ containing $\Gamma(\gamma_0)$ 
with $\pi_{\Theta}^{-1}(P_{\Theta}^{(\gamma_0^+)})=P^{(\gamma_0^+)}$. Then the set $\Lambda(\gamma_0^+)=\bigcup_{\alpha\in\Theta}B(\alpha)$ is a saturated set of cardinality $\gamma_0^+$ which contains $\Gamma(\gamma_0)$ and is contained in $\Gamma(\lambda)$ such that  
$\pi_{\Lambda(\gamma_0^+)}^{-1}(P_{\Lambda(\gamma_0^+)}^{(\gamma_0^+)})=P^{(\gamma_0^+)}$. By Proposition 2.3(ii), we can also suppose that every compact subset of $P_{\Lambda(\gamma_0^+)}\backslash P_{\Lambda(\gamma_0^+)}^{(\gamma_0)}$ is $\gamma_0^+$-negligible in $X_{\Lambda(\gamma_0^+)}$. Suppose the sets $\Lambda(\gamma)$ are constructed for all $\gamma<\beta$ satisfying conditions $(13-(16)$. If $\beta$ is a limit cardinal, then $\beta\in\mathfrak{L}_c$ because $(\gamma_0,\lambda)\subset\mathfrak{L}_c$. Since each $\Lambda(\gamma)$, $\gamma<\beta$, satisfies condition $(15)$, so does the set 
$\Lambda(\beta)=\bigcup_{\gamma<\beta}\Lambda(\gamma)$. If $\beta=\gamma^+$ for some $\gamma$, we define $\Lambda(\beta)=\bigcup_{\alpha\in\Theta(\gamma^+)}B(\alpha)\cup\Lambda(\gamma)$. In this case, increasing $\Lambda(\beta)$ if necesarily, we can assume that the pair $\Lambda(\gamma)\subset\Lambda(\beta)$ satisfies the conditions of Proposition 2.3.
This completes the proof of Claim 4.6.

We are going to define homeomorphisms $g(\gamma)=g_{\Lambda(\gamma)}\in H(X_{\Lambda(\gamma)})$ extending $f_\gamma$, $\gamma\in (\gamma_0,\lambda)$ such that
$\displaystyle\pi_{\Lambda(\gamma_2)}^{\Lambda(\gamma_1)}\circ g(\gamma_1)=
g(\gamma_2)\circ\pi_{\Lambda(\gamma_2)}^{\Lambda(\gamma_1)}$ for $\gamma_0<\gamma_2<\gamma_1<\lambda$ and 
$\displaystyle\pi_{\Lambda(\gamma_0)}^{\Lambda(\gamma)}\circ g(\gamma)=
f_{\gamma_0}\circ\pi_{\Lambda(\gamma_0)}^{\Lambda(\gamma)}$ for $\gamma\in (\gamma_0,\lambda)$.

 The homeomorphism $g(\gamma_0^+)$ exists by Corollary 4.5. Suppose $\beta\in (\gamma_0,\lambda)$ and the homeomorphisms $g(\gamma)$, $\gamma<\beta$, have been defined for all $\gamma\in (\gamma_0,\beta)$. If $\beta$ is a limit cardinal, then  
 by $(14)$, we can define $g(\beta)$ as the limit of all $g(\gamma)$, $\gamma\in (\gamma_0,\beta)$.
If $\beta=\gamma^+$ for some $\gamma>\gamma_0$, the required homeomorphism $g(\beta)$ exists by Corollary 4.5 because the pair $\Lambda(\gamma)\subset\Lambda(\beta)$ satisfies the hypotheses of Proposition 2.3. Finally, we define $\widetilde f_\lambda$ to be the limit of all $g_\gamma$, $\gamma<\lambda$. This definition is correct because $\Gamma(\lambda)=\bigcup_{\gamma_0\leq\gamma<\lambda}\Lambda(\gamma)$.

 Now, consider the case when $P=P^{(\lambda_0)}$ for some $\lambda_0<\tau$.
In this situation we take saturated sets $\Gamma(\lambda)$, $\lambda\leq\lambda_0$ satisfying conditions $(8)-(10)$ with
$P=\pi_{\Gamma(\lambda_0)}^{-1}(P_{\Gamma(\lambda_0)})$. Then $P=P_{\Gamma(\lambda_0)}\times X_{A\backslash\Gamma(\lambda_0)}$. We claim that
the homeomorphism $f_{\lambda_0}\in H(P_{\Gamma(\lambda_0)})$ satisfies the hypotheses of Theorem 1.1, i.e. $f_{\lambda_0}$ preserves the $\lambda$-interiors of $P_{\Gamma(\lambda_0)}$ for all $\lambda\leq\lambda_0$. This is obvious for $\lambda=\lambda_0$ because the weight of $P_{\Gamma(\lambda_0)}$ is $\leq\lambda_0$. Suppose $\lambda<\lambda_0$ is not a limit cardinal. Since
$P^{(\lambda)}=\pi_{\Gamma(\lambda)}^{-1}(P^{(\lambda)}_{\Gamma(\lambda)})$ and
$P^{(\lambda)}_{\Gamma(\lambda_0)}=(\pi^{\Gamma(\lambda_0)}_{\Gamma(\lambda)})^{-1}(P^{(\lambda)}_{\Gamma(\lambda)})$ we have
 $P^{(\lambda)}_{\Gamma(\lambda_0)}\subset (P_{\Gamma(\lambda_0)})^{(\lambda)}$. On the other hand, the inclusion
 $(\pi_{\Gamma(\lambda_0)})^{-1}((P_{\Gamma(\lambda_0)})^{(\lambda)})\subset P^{(\lambda)}$ implies
 $(P_{\Gamma(\lambda_0)})^{(\lambda)}\subset P_{\Gamma(\lambda_0)}^{(\lambda)}$. Hence, $(P_{\Gamma(\lambda_0)})^{(\lambda)}= P_{\Gamma(\lambda_0)}^{(\lambda)}$ and, because $f(P^{(\lambda)})=P^{(\lambda)}$ and $\pi_{\Gamma(\lambda_0)}\circ f=f_{\lambda_0}\circ\pi_{\Gamma(\lambda_0)}$, $f_{\lambda_0}$ preserves the $\lambda$-interior of $P_{\Gamma(\lambda_0)}$. The case when $\lambda<\lambda_0$ is a limit cardinal can be teated in a similar way.
 Hence, we can apply the previous case to extend $f_{\lambda_0}$ to a homeomorphism $\widetilde{f}_{\lambda_0}\in H(X_{\Gamma(\lambda_0)})$. Finally, Corollary 3.5 provides a homeomorphism $\widetilde f\in H(X)$ extending $f$.
   $\Box$

%%%%%%%%%%%%%%%%%%%%%%%%%%%%%%%%%%%%%%%%%%%%%%%%%%%%%%%%%%%%%%%%%%%
%%%%%%%%%%%%%%%%%%%%%%%%%%%%%%%%%%%%%%%%%%%%%%%%%%%%%%%%%%%%%%%%%%%%%%

\end{document}